\documentclass[12pt,reqno,a4]{article}

\usepackage{amsfonts}
\usepackage{amsmath}
\usepackage{amssymb}
\usepackage{amsthm}
\usepackage{enumitem}
\usepackage[pdfstartview=FitV,bookmarks=false]{hyperref}
\usepackage[notref,notcite,final]{showkeys}
\usepackage[notext]{kpfonts}
\usepackage[onlytext]{MinionPro}

\def\blfootnote{\xdef\@thefnmark{}\@footnotetext}
\long\def\symbolfootnote[#1]#2{\begingroup\def\thefootnote{\fnsymbol{footnote}}\footnote[#1]{#2}\endgroup}

\DeclareMathOperator{\Arg}{Arg}
\DeclareMathOperator{\ef}{e}
\DeclareMathOperator{\hf}{h}

\DeclareMathOperator{\Img}{Im}
\DeclareMathOperator{\Log}{Log}
\DeclareMathOperator{\Rl}{Re}

\newcommand{\lt}[1]{\mathcal{L}^{#1}}
\newcommand{\reg}[1]{\mathcal{R}^{#1}}

\newcommand{\crd}[1]{\mathrm{C}_{\mathrm{rd}}^{#1}}
\newcommand{\creg}[1]{\mathcal{R}_{\mathrm{c}}^{#1}}
\newcommand{\nf}{\mathcal{N}}
\newcommand{\C}{\mathbb{C}}
\newcommand{\N}{\mathbb{N}}
\newcommand{\R}{\mathbb{R}}
\newcommand{\T}{\mathbb{T}}
\newcommand{\Z}{\mathbb{Z}}

\newtheorem{theorem}{Theorem}
\newtheorem{definition}{Definition}
\newtheorem{lemma}{Lemma}

\theoremstyle{definition}
\newtheorem{corollary}{Corollary}

\numberwithin{equation}{section}

\begin{document}
\title{On Uniqueness of the Laplace Transform on Time Scales}
\author{Ba\c{s}ak Karpuz}
\date{ }

\maketitle

\symbolfootnote[1]{The author is on 1-year leave at the University of Calgary, Canada.\newline
\textbf{Current Address}. Department of Statistics and Mathematics, University of Calgary, 2500 University Drive N. W., Calgary, AB T2N 1N4, Canada.\newline
\textbf{Address}. Department of Mathematics, Faculty of Science and Literature, ANS Campus, Afyon Kocatepe University, 03200 Afyonkarahisar, Turkey.\newline
\textbf{Email}. bkarpuz@gmail.com \textbf{Web}. \url{http://www2.aku.edu.tr/\string~bkarpuz}}

\begin{abstract}
After introducing the concept of null functions, we shall present a uniqueness result in the
sense of the null functions for the Laplace transform on time scales with arbitrary graininess.
The result can be regarded as a dynamic extension of the well-known Lerch's theorem.
\end{abstract}

\section{Introduction}\label{intro}

The Laplace transform is one of the fundamental representatives of
integral transformations used in mathematical analysis.
This transform is essentially bijective for the majority of
practical uses. The Laplace transform has the useful property that
many relationships and operations over the originals functions correspond to
simpler relationships and operations over the image functions.
The discrete analogue of the Laplace transform is called as $Z$-transform,
which is also converts a sequence of real or complex numbers
into a complex frequency-domain representation.
This transform is also bijective.

The calculus on time scales has been initiated by Hilger (see \cite{MR1066641}) in order to
unify the theories of continuous analysis and of discrete analysis.
The Laplace transform on time scales was introduced by Hilger in \cite{MR1722974},
but in a form that tries to unify the (continuous) Laplace transform
and the (discrete) $Z$-transform.
For arbitrary time scales, the Laplace transform was
introduced by and investigated by Bohner and Peterson in \cite{MR1948468}
(see also \cite[Section~3.10]{MR1843232}).
Let $\sup\T=\infty$, for locally $\Delta$-integrable function $f:[s,\infty)_{\T}\to\C$,
i.e., $\Delta$-integrable over each compact interval of $[s,\infty)_{\T}$, the Laplace transform is defined to be
\begin{equation}
\lt{}\{f\}(z):=\int_{s}^{\infty}f(\eta)\ef_{\ominus{}z}(\sigma(\eta),s)\Delta\eta\quad\text{for}\ z\in\mathcal{D},\notag
\end{equation}
where $\mathcal{D}$ consists of such complex numbers for which the improper integral converges.
In order to determine an explicit region of convergence,
conditions on the class of the determining functions should be provided.
This was done by Davis et al.\ in \cite{MR2324337},
where some restrictions were imposed on the graininess $\mu$.
In a recent paper \cite{bo/gu/ka10}, Bohner et al.\ removed the restriction on the graininess
of the time scale and considered some fundamental properties of the
Laplace transform on time scales.
The readers may be referred to \cite{MR1129464,MR1716143,MR0622023} for the basic properties
of the usual Laplace transform.
For other results about the Laplace transform on time scales,
see \cite{MR2597442,MR2320804,MR2585078,MR2679122}.

The uniqueness property of the Laplace transform and of the $Z$-transform are well-known
(see \cite{MR1129464,MR1716143,MR0622023}), which is a necessary tool for the inverse Laplace transform.
To the best of our knowledge, nothing has been published up to now on the uniqueness of the
Laplace transform on arbitrary time scales.
In this paper, we shall provide a uniqueness result on time scales with arbitrary
graininess for the Laplace transform,
which reduces to the well-known the Lerch's theorem in the continuous case.
Our result on time scales with constant graininess ($\R$ and $\Z$),
gives a unified proof for the uniqueness of the Laplace transform
(the usual Laplace transform and the $Z$-transform).

The paper is organized as follows:
In Section~\ref{al}, we present some results that are required in the proof of
the main result.
In Section~\ref{ult}, we state and prove our main results together with some
necessary definitions.
And in Section~\ref{pts}, as an appendix, we recall a short account concerning
the time scale calculus.

\section{Auxiliary Lemmas}\label{al}

We define the minimal graininess function $\mu_{\ast}:\T\to\R_{0}^{+}$ by
\begin{equation}
\mu_{\ast}(s):=\inf\big\{\mu(t):\ t\in[s,\infty)_{\T}\big\}\quad\text{for}\ s\in\T\notag
\end{equation}
and the set of positively regressive constants $\creg{+}(\T,\C)$ by
\begin{equation}
\creg{+}(\T,\C):=\big\{z\in\C:\ 1+z\mu(t)>0\ \text{for all}\ t\in\T\big\}.\notag
\end{equation}
For $h\in\R_{0}^{+}$ and $\lambda\in\creg{+}(\T,\R)$, we also define the set $\C_{h}(\lambda)$ by
\begin{equation}
\C_{h}(\lambda):=\big\{z\in\C_{h}:\ \Rl_{h}(z)>\lambda\big\}.\notag
\end{equation}

Now, we proceed this section with a result quoted from \cite{bo/gu/ka10}.

\begin{lemma}[See {\protect\cite[Theorem~3.4(iii)]{bo/gu/ka10}}]\label{allm1}
Let $\sup\T=\infty$, $s\in\T$, $\lambda\in\creg{+}([s,\infty)_{\T},\R)$ and $z\in\C_{\mu_{\ast}(s)}(\lambda)$, then
\begin{equation}
\lim_{t\to\infty}\ef_{\lambda\ominus{}z}(t,s)=0.\notag
\end{equation}
\end{lemma}

The inclusion $\R^{+}\subset\C_{\mu_{\ast}(s)}(0)$ for any $s\in\T$ yields the following corollary.

\begin{corollary}\label{alcrl1}
Let $\sup\T=\infty$, $s\in\T$ and $x\in\R^{+}$, then
\begin{equation}
\lim_{t\to\infty}\ef_{\ominus{}x}(t,s)=0\quad\text{and}\quad\lim_{t\to\infty}\ef_{x}(t,s)=\infty.\notag
\end{equation}
\end{corollary}

Next, we present a result on asymptotic property of the time scale exponential.

\begin{lemma}\label{allm2}
Let $s,t\in\T$ and $\lambda\in\R^{+}$, then
\begin{equation}
\lim_{x\to\infty}\big[x^{\lambda}\ef_{\ominus{}x}(t,s)\big]=
\begin{cases}
0,&t>s\\
\infty,&t\leq s.
\end{cases}\notag
\end{equation}
\end{lemma}

\begin{proof}
As we will be considering the limit as $x\to\infty$, we may assume that $x\in\R^{+}$.
First, we consider the case $s,t\in\T$ with $t>s$.
We may find $n\in\N$ such that $n>\lambda$.
By the Taylor's formula, we have
\begin{equation}
\ef_{x}(t,s)=\sum_{\ell=0}^{n}x^{\ell}\hf_{\ell}(t,s)+x^{n+1}\int_{s}^{t}\hf_{n}(t,\sigma(\eta))\ef_{x}(\eta,s)\Delta\eta\geq{}x^{n}\hf_{n}(t,s).\notag
\end{equation}
Therefore, we see that
\begin{equation}
0\leq x^{\lambda}\ef_{\ominus{}x}(t,s)=\frac{x^{\lambda}}{\ef_{x}(t,s)}\leq\frac{x^{\lambda}}{x^{n}\hf_{n}(t,s)},\notag
\end{equation}
which proves $x^{\lambda}\ef_{\ominus{}x}(t,s)\to0$ by letting $x\to\infty$.
Next, let $s,t\in\T$ with $t\leq s$, then $\ef_{\ominus{}x}(t,s)=\ef_{x}(s,t)\geq1$ and thus we have $x^{\lambda}\ef_{\ominus{}x}(t,s)\geq{}x^{\lambda}$, which shows that $x^{\lambda}\ef_{\ominus{}x}(t,s)\to\infty$ as $x\to\infty$.
This completes the proof.
\end{proof}

Let us introduce the function $\Lambda:\creg{}(\T,\R)\times\T\times\T\to\R$ defined by
\begin{equation}
\Lambda(x;t,s):=\exp\big\{-x\ef_{\ominus{}x}(t,s)\big\}\quad\text{for}\ x\in\creg{}(\T,\R)\ \text{and}\ s,t\in\T.\label{aleq1}
\end{equation}

\begin{corollary}\label{alcrl2}
Let $s,t\in\T$, then
\begin{equation}
\lim_{x\to\infty}\Lambda(x;t,s)=\chi_{(-\infty,t)_{\T}}(s),\notag
\end{equation}
where $\chi_{D}:\R\to\{0,1\}$ is the characteristic function of the set $D\subset\R$.
\end{corollary}

\section{Uniqueness of the Laplace Transform}\label{ult}

In this section, we shall always assume that $\sup\T=\infty$.
We first start with the definition of the set of null functions.

\begin{definition}\label{ultdf1}
A function $f:[s,\infty)_{\T}\to\C$ is called a \emph{null function} if
\begin{equation}
\int_{s}^{t}f(\eta)\Delta\eta=0\quad\text{for all}\ t\in[s,\infty)_{\T}.\notag
\end{equation}
The set of null functions on will be denoted by $\nf([s,\infty)_{\T},\C)$.
\end{definition}

Next, we give some properties of the null functions some of
which will be required in the proof of the main result.

\begin{lemma}\label{ultlm2}
Let $f\in\nf([s,\infty)_{\T},\C)$ and $g\in\crd{1}([s,\infty)_{\T},\C)$, then $fg^{\sigma}\in\nf([s,\infty)_{\T},\C)$.
\end{lemma}

\begin{proof}
Performing an integration by parts, for any $t\in[s,\infty)_{\T}$, we have
\begin{equation}
\int_{s}^{t}f(\eta)g^{\sigma}(\eta)\Delta\eta=\Bigg[\bigg[\int_{s}^{\eta}f(\zeta)\Delta\zeta\bigg]g(\eta)\Bigg]_{\eta=s}^{\eta=t}-\int_{s}^{t}\bigg[\int_{s}^{\eta}f(\zeta)\Delta\zeta\bigg]g^{\Delta}(\eta)\Delta\eta=0,\notag
\end{equation}
which proves the claim.
\end{proof}

\begin{corollary}\label{ultcrl1}
Let $f\in\nf([s,\infty)_{\T},\C)$ and $g\in\reg{}([s,\infty)_{\T},\C)$, then $f\ef_{g}(\sigma(\cdot),s)\in\nf([s,\infty)_{\T},\C)$.
\end{corollary}

\begin{corollary}\label{ultcrl1a}
Let $f\in\nf([s,\infty)_{\T},\C)$, then
\begin{equation}
\int_{s}^{\infty}f(\eta)\ef_{\ominus z}(\sigma(\eta),s)\Delta\eta=0\quad\text{for any}\ z\in\creg{}([s,\infty)_{\T},\C).\label{ultcrl1aeq1}
\end{equation}
\end{corollary}

We have now filled the necessary background for the proof of our main result.

\begin{theorem}[Lerch's theorem]\label{lerchthm}
Assume that $f:[s,\infty)_{\T}\to\C$, there exist an increasing divergent sequence $\{\varsigma_{k}\}_{k\in\N_{0}}\subset\R_{0}^{+}$ and $\alpha\in\creg{+}([s,\infty)_{\T},\C)$ such that
\begin{equation}
\lt{}\{f\ef_{\ominus(n\odot\varsigma_{k})}(\sigma(\cdot),s)\}(\alpha)=0\quad\text{for all}\ n,k\in\N_{0}.\label{lerchthmeq1}
\end{equation}
Then $f\in\nf([s,\infty)_{\T},\C)$.
\end{theorem}

\begin{proof}
Define the function $g:[s,\infty)_{\T}\to\C$ by $g(t):=f(t)\ef_{\ominus\alpha}(\sigma(t),s)$ for $t\in[s,\infty)_{\T}$, then we have
\begin{align}
\int_{s}^{\infty}g(\eta)\ef_{\ominus(n\odot\varsigma_{k})}(\sigma(\eta),s)\Delta\eta=&\int_{s}^{\infty}f(\eta)\ef_{\ominus(n\odot\varsigma_{k})}(\sigma(\eta),s)\ef_{\ominus\alpha}(\sigma(\eta),s)\Delta\eta\notag\\
=&\int_{s}^{\infty}f(\eta)\ef_{\ominus(\alpha\oplus(n\odot\varsigma_{k}))}(\sigma(\eta),s)\Delta\eta=0\label{lerchthmprfeq1}
\end{align}
for all $n,k\in\N_{0}$.
Let $r\in[s,\infty)_{\T}$, and define $h:[s,\infty)_{\T}\to\C$ by
\begin{equation}
h_{r}(t):=\int_{r}^{t}g(\eta)\Delta\eta\quad\text{for}\ t\in[s,\infty)_{\T}.\notag
\end{equation}
It follows from \eqref{lerchthmeq1} with $n=0$ that
\begin{equation}
\int_{s}^{\infty}g(\eta)\Delta\eta=0,\label{lerchthmprfeq2}
\end{equation}
which shows that $\lim_{t\to\infty}h_{r}(t)$ exists.
So we can find $M_{r}\in\R^{+}$ such that $|h_{r}(t)|\leq M_{r}$ for all $t\in[r,\infty)_{\T}$.
We may (and do) assume that $M_{r}\to0$ as $r\to\infty$.
Using \eqref{lerchthmprfeq1}, and performing integration by parts, we get
\begin{align}
\int_{s}^{r}g(\eta)\ef_{\ominus(n\odot\varsigma_{k})}(\sigma(\eta),s)\Delta\eta=&-\int_{r}^{\infty}g(\eta)\ef_{\ominus(n\odot\varsigma_{k})}(\sigma(\eta),s)\Delta\eta\notag\\
\begin{split}
=&-\bigg[\big[\ef_{\ominus(n\odot\varsigma_{k})}(\eta,s)h_{r}(\eta)\big]_{\eta=r}^{\eta\to\infty}\\
&-\int_{r}^{\infty}\ef_{\ominus(n\odot\varsigma_{k})}^{\Delta_{1}}(\eta,s)h_{r}(\eta)\Delta\eta\bigg]\\
\end{split}\notag\\
\begin{split}
=&-\bigg[\big[\big(\ef_{\ominus\varsigma_{k}}(\eta,s)\big)^{n}h_{r}(\eta)\big]_{\eta=r}^{\eta\to\infty}\\
&+\int_{r}^{\infty}(n\odot\varsigma_{k})(\eta)\ef_{\ominus(n\odot\varsigma_{k})}(\sigma(\eta),s)h_{r}(\eta)\Delta\eta\bigg]
\end{split}\notag\\
=&-\int_{r}^{\infty}(n\odot\varsigma_{k})(\eta)\ef_{\ominus(n\odot\varsigma_{k})}(\sigma(\eta),s)h_{r}(\eta)\Delta\eta\label{lerchthmprfeq3}
\end{align}
for all $n,k\in\N_{0}$.
Note that above have used Corollary~\ref{alcrl1} while passing to the last step.
Now multiplying both sides of \eqref{lerchthmprfeq3} with $\ef_{\ominus(n\odot\varsigma_{k})}(s,r)$, we have
\begin{equation}
\int_{s}^{r}g(\eta)\ef_{\ominus(n\odot\varsigma_{k})}(\sigma(\eta),r)\Delta\eta=-\int_{r}^{\infty}(n\odot\varsigma_{k})(\eta)\ef_{\ominus(n\odot\varsigma_{k})}(\sigma(\eta),r)h_{r}(\eta)\Delta\eta,\notag
\end{equation}
which yields
\begin{equation}
\bigg|\int_{s}^{r}g(\eta)\ef_{\ominus(n\odot\varsigma_{k})}(\sigma(\eta),r)\Delta\eta\bigg|\leq M_{r}\int_{r}^{\infty}(n\odot\varsigma_{k})(\eta)\ef_{\ominus(n\odot\varsigma_{k})}(\sigma(\eta),r)\Delta\eta=M_{r}.\notag
\end{equation}
By the series expansion of the exponential function, we know that
\begin{equation}
\sum_{\ell\in\N_{0}}\frac{(-1)^{\ell}\varsigma_{k}^{\ell}}{\ell!}\ef_{\ominus(\ell\odot\varsigma_{k})}(t,s)=\Lambda(\varsigma_{k};t,s)\quad\text{for}\ s,t\in\T\ \text{and}\ k\in\N_{0},\notag
\end{equation}
where $\Lambda$ is defined by \eqref{aleq1}.
Thus, for all $t\in[s,r)_{\T}$ and all $k\in\N_{0}$, we can estimate that
\begin{align}
\bigg|\int_{s}^{r}g(\eta)\Lambda(\varsigma_{k};\sigma(\eta),t)\Delta\eta\bigg|=&\bigg|\sum_{\ell\in\N_{0}}\frac{(-1)^{\ell-1}\varsigma_{k}^{\ell}}{\ell!}\int_{s}^{r}g(\eta)\ef_{\ominus(\ell\odot\varsigma_{k})}(\sigma(\eta),t)\Delta\eta\bigg|\notag\\
\leq&\sum_{\ell\in\N_{0}}\frac{\varsigma_{k}^{\ell}}{\ell!}\ef_{\ominus(\ell\odot\varsigma_{k})}(r,t)\bigg|\int_{s}^{r}g(\eta)\ef_{\ominus(\ell\odot\varsigma_{k})}(\sigma(\eta),r)\Delta\eta\bigg|\notag\\
\leq&M_{r}\sum_{\ell\in\N_{0}}\frac{\varsigma_{k}^{\ell}}{\ell!}\big(\ef_{\ominus\varsigma_{k}}(r,t)\big)^{\ell}=M_{r}\exp\big\{\varsigma_{k}\ef_{\ominus\varsigma_{k}}(r,t)\big\}.\notag
\end{align}
Letting $r\to\infty$, we have $M_{r}\to0$ and $\ef_{\ominus\varsigma_{k}}(r,t)\to0$ by Corollary~\ref{alcrl1}, which yields $M_{r}\exp\{\varsigma_{k}\ef_{\ominus\varsigma_{k}}(r,t)\}\to0$ as $r\to\infty$.
We can therefore write
\begin{equation}
\int_{s}^{\infty}g(\eta)\Lambda(\varsigma_{k};\sigma(\eta),t)\Delta\eta=0\quad\text{for all}\ k\in\N_{0}.\label{lerchthmprfeq4}
\end{equation}
By \eqref{lerchthmprfeq2}, the function $g$ is integrable over $[s,\infty)_{\T}$ and the characteristic function $\chi$ is piecewise constant.
Letting $k\to\infty$ in \eqref{lerchthmprfeq4}, we get
\begin{equation}
\int_{s}^{\infty}g(\eta)\chi_{(-\infty,t)_{\T}}(\sigma(\eta))\Delta\eta=0\quad\text{for all}\ t\in[s,\infty)_{\T}\label{lerchthmprfeq5}
\end{equation}
by Lebesque's dominated convergence theorem and Corollary~\ref{alcrl2}.
Now, we are in a position to prove that
\begin{equation}
\int_{s}^{t}g(\eta)\Delta\eta=0\quad\text{for all}\ t\in[s,\infty)_{\T}.\label{lerchthmprfeq6}
\end{equation}
From \eqref{lerchthmprfeq5}, for all $t\in[s,\infty)_{\T}$, we have
\begin{align}
\int_{s}^{\infty}g(\eta)\chi_{(-\infty,t)_{\T}}(\sigma(\eta))\Delta\eta=&\int_{s}^{t}g(\eta)\chi_{[s,t)_{\T}}(\sigma(\eta))\Delta\eta\notag\\
=&\int_{s}^{t}g(\eta)\big[\chi_{[s,t)_{\T}}(\eta)+\mu(\eta)\chi_{[s,t)_{\T}}^{\Delta}(\eta)\big]\Delta\eta\notag\\
=&\int_{s}^{t}g(\eta)\chi_{[s,t)_{\T}}(\eta)\Delta\eta,\notag
\end{align}
which together with the definition of the characteristic function $\chi$ and \eqref{lerchthmprfeq5} gives \eqref{lerchthmprfeq6}.
Therefore, we learn that $g$ is a null function.
An application of Corollary~\ref{ultcrl1} shows that $f=g\ef_{\alpha}(\sigma(\cdot),s)$ is a null function too.
This completes the proof.
\end{proof}

\begin{corollary}
Assume that $f,g:[s,\infty)_{\T}\to\C$, there exist an increasing divergent sequence $\{\varsigma_{k}\}_{k\in\N_{0}}\subset\R^{+}$ and $\alpha\in\creg{+}([s,\infty)_{\T},\C)$ such that
\begin{equation}
\lt{}\{f\ef_{\ominus(n\odot\varsigma_{k})}(\sigma(\cdot),s)\}(\alpha)=\lt{}\{g\ef_{\ominus(n\odot\varsigma_{k})}(\sigma(\cdot),s)\}(\alpha)\quad\text{for all}\ n,k\in\N_{0}.\notag
\end{equation}
Then $f-g\in\nf([s,\infty)_{\T},\C)$.
\end{corollary}

\begin{corollary}
Assume that the graininess function $\mu$ is constant and there exists $\alpha\in\creg{+}([s,\infty)_{\T},\C)$ such that
\begin{equation}
\lt{}\{f\}(z)=0\quad\text{for all}\ z\in\C_{\mu_{\ast}(s)}(\alpha).\notag
\end{equation}
Then $f\in\nf([s,\infty)_{\T},\C)$.
\end{corollary}

\begin{proof}
In this case, for any fixed $\beta\in\R_{\mu_{\ast}(s)}(\alpha)\subset\C_{\mu_{\ast}(s)}(\alpha)$, we have
\begin{equation}
\lt{}\{f\}(\beta)=0,\notag
\end{equation}
which yields $\beta\oplus((nk)\odot\varsigma)\in\R_{\mu_{\ast}(s)}(\alpha)\subset\C_{\mu_{\ast}(s)}(\alpha)$ for all $n,k\in\N_{0}$ and all $t\in[s,\infty)_{\T}$, where $\varsigma\in\R^{+}$, i.e.,
\begin{equation}
\lt{}\{f\}(\beta\oplus((nk)\odot\varsigma))=\lt{}\{f\ef_{\ominus((nk)\odot\varsigma)}(\sigma(\cdot),s)\}(\beta)=0\quad\text{for all}\ n,k\in\N_{0}.\notag
\end{equation}
This shows that the conditions of Theorem~\ref{lerchthm} hold with $\varsigma_{k}:=k\odot\varsigma$ for $k\in\N_{0}$.
\end{proof}

\section{Appendix: Time Scales Essentials}\label{pts}

A \emph{time scale}, which inherits the standard topology on $\R$,
is a nonempty closed subset of reals.
Throughout this paper, the time scale is assumed to be unbounded above and
will be denoted by the symbol $\T$,
and the intervals with a subscript $\T$ are used to denote
the intersection of the usual interval with $\T$.
For $t\in\T$, we define the \emph{forward jump operator}
$\sigma:\T\to\T$ by $\sigma(t):=\inf(t,\infty)_{\T}$
while the \emph{graininess function}
$\mu:\T\to\R_{0}^{+}$ is defined to be $\mu(t):=\sigma(t)-t$.
A point $t\in\T$ is called \emph{right-dense} if $\sigma(t)=t$;
otherwise, it is called \emph{right-scattered},
and similarly \emph{left-dense} and \emph{left-scattered} points
are defined in terms of the so-called backward jump operator.
A function $f:\T\to\C$ is said to be \emph{Hilger differentiable}
(or $\Delta$-differentiable) at the point $t\in\T$ if there exists $\ell\in\C$ such that for any
$\varepsilon>0$ there exists a neighborhood $U$ of $t$ such that
\begin{equation}
\big|[f(\sigma(t))-f(s)]-\ell[\sigma(t)-s]\big|\leq\varepsilon|\sigma(t)-s|\quad\text{for all}\ s\in U,\notag
\end{equation}
and in this case we denote $f^{\Delta}(t)=\ell$.
A function $f$ is called \emph{rd-continuous}
provided that it is continuous at right-dense points in $\T$,
and has finite limits at left-dense points,
and the set of rd-continuous functions is denoted by $\crd{}(\T,\C)$.
The set of functions $\crd{1}(\T,\C)$ includes the functions whose
derivative is in $\crd{}(\T,\C)$ too.
For $f\in\crd{1}(\T,\C)$, we have
\begin{equation}
f^{\sigma}=f+\mu f^{\Delta}\quad\text{on}\ \T^{\kappa},\notag
\end{equation}
where $f^{\sigma}:=f\circ\sigma$ and $\T^{\kappa}:=\T\backslash\{\sup\T\}$ if $\sup\T=\max\T$ and satisfies $\rho(\max\T)\neq\max\T$; otherwise, $\T^{\kappa}:=\T$.
For $s,t\in\T$ and a function $f\in\crd{}(\T,\C)$,
the $\Delta$-integral of $f$ is defined by
\begin{equation}
\int_{s}^{t}f(\eta)\Delta\eta=F(t)-F(s)\quad\text{for}\ s,t\in\T,\notag
\end{equation}
where $F:\T\to\C$ is an antiderivative of $f$,
i.e., $F^{\Delta}=f$ on $\T^{\kappa}$.

A function $f\in\crd{}(\T,\C)$
is called \emph{regressive} if $1+\mu f\neq0$ on $\T$,
and \emph{positively regressive} if it is real valued and $1+\mu f>0$ on $\T$.
The set of regressive functions and the set of positively regressive functions
are denoted by $\reg{}(\T,\C)$ and $\reg{+}(\T,\R)$, respectively,
and $\reg{-}(\T,\R)$ is defined similarly.
For simplicity, we denote by $\creg{}(\T,\C)$ the set of complex regressive constants,
and similarly, we define the sets $\creg{+}(\T,\R)$ and $\creg{-}(\T,\R)$.

Let $f\in\reg{}(\T,\C)$.
Then the \emph{exponential function} $\ef_{f}(\cdot,s)$ is defined
to be the unique solution of the initial value problem
\begin{equation}
\begin{cases}
x^{\Delta}=fx\quad\text{on}\ \T^{\kappa}\\
x(s)=1
\end{cases}\notag
\end{equation}
for some fixed $s\in\T$.
For $h>0$, set
\begin{equation}
\C_{h}:=\big\{z\in\C:\ z\neq-1/h\big\}\quad\text{and}\quad\Z_{h}:=\big\{z\in\C:\-\pi/h<\Img(z)\leq\pi/h\big\},\notag
\end{equation}
and $\C_{0}:=\Z_{0}:=\C$.
For $h\in\R_{0}^{+}$, the Hilger \emph{real part} and \emph{imaginary part}
of a complex number are given by
\begin{equation}
\Rl_{h}(z):=\lim_{\nu\to h}\frac{1}{\nu}\big(|1+\nu z|-1\big)\quad\text{and}\quad\Img_{h}(z):=\lim_{\nu\to h}\frac{1}{\nu}\Arg(1+\nu z),\notag
\end{equation}
respectively, where $\Arg$ denotes the principle argument function,
i.e., $\Arg:\C\to(-\pi,\pi]_{\R}$.
For $h\in\R_{0}^{+}$ and any fixed $z\in\C_{h}$, the Hilger real part $\Rl_{h}(z)$ is a nondecreasing function of $h\in\R_{0}^{+}$, i.e., $\Rl_{h_{1}}(z)\geq\Rl_{h_{2}}(z)$ for $h_{1},h_{2}\in\R_{0}^{+}$ with $h_{1}\geq h_{2}$.
For $h\in\R_{0}^{+}$, we define the \emph{cylinder transformation}
$\xi_{h}:\C_{h}\to\Z_{h}$ by
\begin{equation}
\xi_{h}(z):=\lim_{\nu\to h}\frac{1}{\nu}\Log(1+\nu z)\quad\text{for}\ z\in\C_{h}.\notag
\end{equation}
Then the exponential function can also be written in the form
\begin{equation}
\ef_{f}(t,s):=\exp\bigg\{\int_{s}^{t}\xi_{\mu(\eta)}\big(f(\eta)\big)\Delta\eta\bigg\}\quad\text{for}\ s,t\in\T.\notag
\end{equation}

It is known that the exponential function $\ef_{f}(\cdot,s)$ is
strictly positive on $[s,\infty)_{\T}$
provided that $f\in\reg{+}([s,\infty)_{\T},\R)$,
while $\ef_{f}(\cdot,s)$ alternates in sign at
right-scattered points of the interval $[s,\infty)_{\T}$
provided that $f\in\reg{-}([s,\infty)_{\T},\R)$.
For $h\in\R_{0}^{+}$ and $w,z\in\C_{h}$, the \emph{circle plus} and the \emph{circle minus}
are defined by
\begin{equation}
z\oplus_{h}w:=z+w+hzw\quad\text{and}\quad z\ominus_{\mu}w:=\frac{z-w}{1+hw},\notag
\end{equation}
respectively. It is known that $(\reg{}(\T,\C),\oplus_{\mu})$ is a group,
and the inverse of $f\in\reg{}(\T,\C)$ is $\ominus_{\mu}f:=0\ominus_{\mu}f$.
Moreover, $\creg{+}(\T,\C)$ is a subgroup of $\creg{}(\T,\C)$.
For $\lambda\in\C$ and $z\in\C_{h}$, the \emph{circle dot} is defined by
\begin{equation}
\lambda\odot_{h}z:=\lim_{\nu\to h}\frac{1}{\nu}\big((1+\nu z)^{\lambda}-1\big).\notag
\end{equation}
With this multiplication, $(\reg{}(\T,\C),\oplus_{\mu},\odot_{\mu})$ becomes a complex vector space.
It should be noted that
\begin{equation}
\ef_{\lambda\odot_{\mu}f}(t,s)=\big(\ef_{f}(t,s)\big)^{\lambda}\quad\text{for}\ s,t\in\T,\notag
\end{equation}
where $\lambda\in\C$ and $f\in\reg{}(\T,\C)$.
For simplicity in the notation, we shall use $\oplus,\ominus$ and $\odot$ instead of $\oplus_{\mu},\ominus_{\mu}$ and $\odot_{\mu}$, respectively.

The definition of the generalized monomials on time scales
(see \cite[\S~1.6]{MR1843232}) $\hf_{n}:\T\times\T\to\R$ is given as
\begin{equation}
\hf_{n}(t,s):=
\begin{cases}
1,&n=0,\\
\displaystyle\int_{s}^{t}\hf_{n-1}(\eta,s)\Delta\eta,&n\in\N
\end{cases}\quad\text{for}\ s,t\in\T.\notag
\end{equation}
Using induction, it is easy to see that
$\hf_{n}(t,s)\geq0$ holds for all $n\in\N_{0}$
and all $s,t\in\T$ with $t\geq s$,
and $(-1)^{n}\hf_{n}(t,s)\geq 0$ holds for all $n\in\N$
and all $s,t\in\T$ with $t\leq s$.

The readers are referred to \cite{MR1843232}
for fundamentals of time scale theory.

\end{document}